\documentclass[10pt]{amsart}
\usepackage{amsmath}
\usepackage{amsfonts}
\usepackage{amsthm}
\usepackage{amssymb}
\usepackage{xcolor,color}

\newtheorem{thm}{Theorem}
\newtheorem{lem}[thm]{Lemma}
\newtheorem{prop}[thm]{Proposition}

\newtheorem{remark}[thm]{Remark}

\numberwithin{equation}{section}
\numberwithin{thm}{section}

\newcommand{\Rm}{\mathrm{Rm}}
\newcommand{\Rc}{\mathrm{Rc}}

\newcommand{\End}{\mathrm{End}}

\title[Preservation of product structures]{Preservation of product structures under the Ricci flow with instantaneous curvature bounds}
\author{Mary Cook}

\begin{document}
\begin{abstract}
In this note, we prove that there exists a constant $\epsilon  >0$, depending only on the dimension, such that if a complete solution to the Ricci flow splits as a product at time $t=0$ and has curvature bounded by $\frac{\epsilon}{t}$, then the solution splits for all time.
\end{abstract}
\maketitle

\section{Introduction}

In this note, we consider the problem of whether a solution to the Ricci flow
\begin{equation}\label{rf}
\frac{\partial}{\partial t} g = -2 \Rc
\end{equation}
which splits as a product at $t=0$ continues to do so for all time.

This problem is closely related, but not strictly equivalent, to the question of uniqueness of solutions to (\ref{rf}). For example, when $(\hat{M} \times \check{M},\hat{g}_0 \oplus \check{g}_0)$, Shi's existence theorem \cite{shideformmetric} implies that there exist complete, bounded curvature solutions $(\hat{M}, \hat{g}(t))$ and $(\check{M}, \check{g}(t))$ with initial conditions $\hat{g}_0$ and $\check{g}_0$, respectively, which exist on some common time interval $[0,T]$. Then, $\hat{g} (t) \oplus \check{g}(t)$ solves (\ref{rf}) on $\hat{M} \times \check{M}$ for $t \in [0,T]$ and is also complete and of bounded curvature. But, according to the uniqueness results of Hamilton \cite{ham3d} and Chen-Zhu \cite{chenzhu}, such a solution is unique among those which are complete and have bounded curvature. Thus, any solution in that class starting at $\hat{g}_0 \oplus \check{g}_0$ continues to split as a product.

Outside of this class, less is known. While there are elementary examples which show that without completeness, a solution may instantaneously cease to be a product, the extent to which the uniform curvature bound can be relaxed is less well-understood. (One exception is in dimension two, where the work of Giesen and Topping \cite{gtrfexistence,gtrfnegativelycurved} has established an essentially complete theory of existence and uniqueness for (\ref{rf}). In particular, in \cite{toppinguniqueness}, Topping shows that any two complete solutions with the same initial data must agree.)

One class of particular interest is that of solutions satisfying a curvature bound of the form $c/t$ for some constant $c$, which arise naturally as limits of exhaustions (see, e.g., \cite{cabezasrivaswilking}, \cite{hochard}, \cite{simontopping}). The purpose of this note is to prove the following.

\begin{thm}\label{main}
Let $(\hat{M}, \hat{g}_0)$ and $(\check{M}, \check{g}_0)$ be two connected Riemannian manifolds and let $M = \hat{M} \times \check{M}$ and $g_0 = \hat{g}_0 \oplus \check{g}_0$. Then there exists a constant $\epsilon = \epsilon(n) > 0$, where $n = \mathrm{dim} (M)$, such that if $g(t)$ is a complete solution to (\ref{rf}) on $M \times [0,T]$ with $g(0) = g_0$ satisfying
\begin{equation}\label{bound}
|\Rm| \le \frac{\epsilon}{t},
\end{equation}
then $g(t)$ splits as a product for all $t \in [0,T]$, i.e., there exist $\hat{g}(t), \check{g}(t)$ such that $g(t) = \hat{g}(t) \oplus \check{g}(t)$, where $\hat{g}(t)$ and $\check{g}(t)$ are solutions to (\ref{rf}) on $\hat{M}$ and $\check{M}$, respectively, for $t \in [0,T]$.
\end{thm}

Lee \cite{lee} has already established the uniqueness of complete solutions satisfying the bound (\ref{bound}). However, his result does not directly imply Theorem \ref{main}: without any restrictions on the curvatures of $\hat{g}_0$ and $\check{g}_0$, we lack the short-time existence theory to guarantee that there are \emph{any} solutions on $\hat{M}$ and $\check{M}$, respectively, with the given initial data, let alone solutions satisfying a bound of the form (\ref{bound}) for sufficiently small $\epsilon$. Thus we are unable to construct a competing product solution on $\hat{M} \times \check{M}$ to which we might apply Lee's theorem.

Instead, we frame the problem as one of uniqueness for a related system, using a perspective similar to that of \cite{liuszek} and \cite{kholonomy}. The key ingredient is a maximum principle closely based on one due to Huang-Tam \cite{huangtam} and modified by Liu-Sz\'ekelyhidi \cite{liuszek}. These references establish, among other things, related results concerning the preservation of K\"ahler structures.

\section{Tracking the product structure}

Our first step toward proving Theorem 1 is to construct a system associated to a solution to Ricci flow which measures the degree to which a solution which initially splits as a product fails to remain a product. Consider a Riemannian product $(M,g_0) = (\hat{M} \times \check{M}, \hat{g}_0 \oplus \check{g}_0)$, and let $g(t)$ be a smooth solution to the Ricci flow on $M \times [0,T]$ with $g(0) = g_0$. For the time being, we make no assumptions on the completeness of $g(t)$ or bounds on its curvature. 

\subsection{Extending the projections}

Let $\hat{\pi} : M \to \hat{M}$ and $\check{\pi} : M \to \check{M}$ be the projections on each factor, and let $\hat{H} = \ker( d\check{\pi} )$ and $\check{H} = \ker( d\hat{\pi})$.  We define $\hat{P}_0, \check{P}_0 \in \mathrm{End} (TM)$ to be the orthogonal projections onto $\hat{H}$ and $\check{H}$ determined by $g_0$.

Following \cite{kholonomy}, we extend each of them to a time-dependent family of projections for $t\in [0,T]$ by solving the fiber-wise ODEs
\begin{equation}\label{pdefs}
 \begin{cases}
  \partial_t \hat{P}(t) = \Rc \circ \hat{P} - \hat{P} \circ \Rc\\
  \hat{P}(0) = \hat{P}_0\\
 \end{cases},
 \qquad
 \begin{cases}
  \partial_t \check{P}(t) = \Rc \circ \check{P} - \check{P} \circ \Rc\\
  \check{P}(0) = \check{P}_0\\
 \end{cases}.
\end{equation}
From $\hat{P}$ and $\check{P}$, we construct time-dependent endomorphisms $\mathcal{P}, \bar{\mathcal{P}} \in  \End (\Lambda^2 T^*M)$ by
\begin{equation*}
\mathcal{P}\omega (X,Y) = \omega(\hat{P} X, \check{P} Y) + \omega( \check{P} X, \hat{P} Y),
\end{equation*}
\begin{equation*}
\bar{\mathcal{P}} \omega (X,Y) = \omega( \hat{P} X, \hat{P} Y) + \omega ( \check{P} X, \check{P} Y).
\end{equation*}

Let $\Rm: \Lambda^2 T^*M \to \Lambda^2 T^*M$ be the curvature operator, and define the following:
\begin{equation*}\label{rdef}
\mathcal{R} = \Rm \circ \mathcal{P}, \qquad \bar{\mathcal{R}} = \Rm \circ \bar{\mathcal{P}},
\end{equation*}
\begin{equation*}
\mathcal{S} = (\nabla \Rm) \circ (\mathrm{Id} \times \mathcal{P}), \qquad \mathcal{T} = (\nabla \nabla \Rm) \circ (\mathrm{Id} \times \mathrm{Id} \times \mathcal{P}).
\end{equation*}

In order to study the evolution of $\mathcal{R}$, it will be convenient to introduce an operator $\Lambda^a_b$ which acts algebraically on tensors via
\begin{equation*}
\Lambda^a_b A^{j_1 \dots j_k}_{i_1 \dots i_l} = \delta^a_{i_1} A^{j_1 \dots j_k}_{b i_2 \dots i_l} + \dots + \delta^a_{i_l} A^{j_1 \dots j_k}_{ i_1 \dots b} - \delta^{j_1}_{b} A^{a \dots j_k}_{i_1 \dots i_l} - \dots - \delta^{j_k}_{b} A^{j_1 \dots a}_{i_1 \dots i_l}.
\end{equation*}
We will also consider the operator
\begin{equation*}
D_t := \partial_t + R_{ab} g^{bc} \Lambda^a_c.
\end{equation*}
This operator has the property that $D_t g = 0$, and for any time-dependent tensor fields $A$ and $B$,
\begin{equation*}
D_t \langle A, B \rangle = \langle D_t A,B \rangle + \langle A, D_t B \rangle,
\end{equation*}
where $\langle \cdot, \cdot \rangle$ is the metric induced by $g(t)$.
Note that by construction the projections satisfy
\begin{equation*}
D_t \hat{P} \equiv 0, \quad D_t \check{P} \equiv 0, \quad D_t \mathcal{P} \equiv 0, \quad D_t \bar{\mathcal{P}} \equiv 0 .
\end{equation*}

\subsection{Evolution equations}

In order to determine how the components of $\mathbf{X}$ and $\mathbf{Y}$ evolve, we will make use of the following commutation formulas (see \cite{kholonomy}, Lemma 4.3):
\begin{equation}\label{dtcomm}
[D_t, \nabla_a] = \nabla_p R_{pacb} \Lambda^b_c + R_{ac} \nabla_c,
\end{equation}
\begin{equation}\label{heatcomm}
[D_t - \Delta, \nabla_a] = 2R_{abdc} \Lambda^c_d \nabla_b + 2 R_{ab} \nabla_b.
\end{equation}

Additionally, we will need to examine the sharp operator on endomorphisms of two forms. For any $A, B \in \mathrm{End} (\Lambda^2 T^*M)$, 
\begin{equation*}
\langle A \# B (\varphi), \psi \rangle= \frac{1}{2} \sum_{\alpha, \beta} \langle [A(\omega_\alpha), B(\omega_\beta)], \varphi \rangle \cdot \langle [\omega_\alpha, \omega_\beta], \psi\rangle,
\end{equation*}
where $\varphi, \psi \in \Lambda^2 T^*M$ and $\{ \omega_\alpha \}$ is an orthonormal basis for $\Lambda^2 T^*M$. Recall that the curvature operator evolves according to
\begin{equation*}
(D_t - \Delta) \Rm = \mathcal{Q}(\Rm, \Rm),
\end{equation*}
under the Ricci flow, where $\mathcal{Q} (A,B) = \frac{1}{2}(AB + BA) + A \# B$.

\begin{prop}\label{schematicequationsp}
We have the following evolution equations for the projection $\mathcal{P}$:
\begin{equation*}
D_t \nabla \mathcal{P} = \Rm *\nabla \mathcal{P} + \mathcal{P} * \mathcal{S},
\end{equation*}
\begin{equation*}
D_t \nabla^2 \mathcal{P} = \Rm * \nabla^2 \mathcal{P} + \nabla \Rm * \nabla \mathcal{P} + \mathcal{P} * \mathcal{T} + \nabla \Rm * \mathcal{P} * \nabla \mathcal{P}.
\end{equation*}
In particular, there exists a constant $C = C(n)$ such that
\begin{equation}\label{pnorms}
\begin{split}
|D_t \nabla \mathcal{P}| &\le C ( |\Rm| |\nabla \mathcal{P}| +|\mathcal{S}|), \\
|D_t \nabla^2 \mathcal{P}| & \le C( |\Rm| |\nabla^2 \mathcal{P}| + |\nabla \Rm| | \nabla \mathcal{P}| + |\mathcal{T}|).\\
\end{split}
\end{equation}
\end{prop}

\begin{proof}

Using equation (\ref{dtcomm}) and the fact that $D_t \mathcal{P}=0$, we can see that $D_t \nabla \mathcal{P} = [D_t, \nabla] \mathcal{P}$.
With some additional computation, we can then see (as in Propositions 4.5 and 4.6 from \cite{kholonomy}) that

\begin{equation*}
D_t \nabla \mathcal{P} = \Rm *\nabla \mathcal{P} + \mathcal{P} * \mathcal{S}.
\end{equation*}
Similarly, using this equation together with (\ref{dtcomm}) and the fact that $\nabla \mathcal{S} = \mathcal{T} + \nabla \Rm * \nabla \mathcal{P}$, we have

\begin{equation*}
\begin{split}
D_t \nabla^2 \mathcal{P}& = [D_t, \nabla] \nabla \mathcal{P} + \nabla( D_t \nabla \mathcal{P})\\
& = \Rm * \nabla^2 \mathcal{P} + \nabla \Rm * \nabla \mathcal{P} + \mathcal{P} * \mathcal{T} + \nabla \Rm * \mathcal{P} * \nabla \mathcal{P},\\
\end{split}
\end{equation*}
as claimed.
\end{proof}

In order compute similar evolution equations for $\mathcal{R}, \mathcal{S},$ and $\mathcal{T}$, we will need the following lemma.

\begin{lem}\label{q}
Let $A, B \in \mathrm{End} (\Lambda^2 T^*M)$ be self-adjoint operators. There exists $C=C(n)>0$ such that
\begin{equation*}
| \mathcal{Q}(A,B) \circ \mathcal{P}| \le C \left( | A \circ \mathcal{P} | |B| + | A | | B \circ \mathcal{P}| \right).
\end{equation*}
\end{lem}

\begin{proof}
Clearly,
\begin{equation*}
| (A \circ B + B \circ A) \circ \mathcal{P}| \le | A \circ \mathcal{P} | |B| + | A | | B \circ \mathcal{P}|.
\end{equation*}
Furthermore, for $\eta \in \Lambda^2 T^*M$,
\begin{equation*}
\begin{split}
\big( (A \# B)\circ \mathcal{P} \big) (\eta) & = \frac{1}{2} \sum_{\alpha, \beta} \langle [ A \omega_\alpha, B \omega_\beta] , \mathcal{P} \eta \rangle \cdot [ \omega_\alpha, \omega_\beta ] \\
& = \frac{1}{2} \sum_{\alpha, \beta} \langle [ \mathcal{P} \circ A \omega_\alpha, \mathcal{P} \circ B \omega_\beta] , \mathcal{P} \eta \rangle \cdot [ \omega_\alpha, \omega_\beta ] \\
& \quad + \frac{1}{2} \sum_{\alpha, \beta} \langle [ \bar{\mathcal{P}} \circ A \omega_\alpha, \mathcal{P} \circ B \omega_\beta] , \mathcal{P} \eta \rangle \cdot [ \omega_\alpha, \omega_\beta ] \\
& \quad  + \frac{1}{2} \sum_{\alpha, \beta} \langle [ \mathcal{P} \circ A \omega_\alpha, \bar{\mathcal{P}} \circ B \omega_\beta] , \mathcal{P} \eta \rangle \cdot [ \omega_\alpha, \omega_\beta ] \\
&\quad +  \frac{1}{2} \sum_{\alpha, \beta} \langle [ \bar{\mathcal{P}} \circ A \omega_\alpha, \bar{\mathcal{P}} \circ B \omega_\beta] , \mathcal{P} \eta \rangle \cdot [ \omega_\alpha, \omega_\beta ],\\
\end{split}
\end{equation*}
where $\{ \omega_\alpha\}$ is an orthonormal basis for $\Lambda^2 T^*M$. The final term on the right hand side is zero (see \cite{kholonomy}, Lemma 3.5); the point is that the image of $\bar{\mathcal{P}}$ is closed under the bracket and is perpendicular to the image of $\mathcal{P}$. Moreover, $\mathcal{P} \circ A = (A \circ \mathcal{P})^*$ and $\mathcal{P} \circ B = (B \circ \mathcal{P})^*$, so it follows that
\begin{equation*}
\begin{split}
| (A \# B) \circ \mathcal{P}| & \le C \big( | A \circ \mathcal{P}| |B \circ \mathcal{P}| + | A \circ \bar{\mathcal{P}} | | B \circ \mathcal{P} | + |A \circ \mathcal{P}| | B \circ \bar{\mathcal{P}}| \big)\\
& \le C \big( |A \circ \mathcal{P}| |B| + |A| |B \circ \mathcal{P}| \big),\\
\end{split}
\end{equation*}
completing the proof.
\end{proof}

\begin{prop}\label{schematicequationsr}
As defined above, $\mathcal{R}$, $\mathcal{S}$, and $\mathcal{T}$ satisfy the inequalities

\begin{equation}\label{rnorms}
\begin{split}
|(D_t - \Delta) \mathcal{R}| & \le C(|\Rm| | \mathcal{R}| + |\nabla \Rm| |\nabla \mathcal{P}| + |\Rm| |\nabla^2 \mathcal{P}| ), \\
|(D_t - \Delta) \mathcal{S}| & \le C( |\nabla \Rm| |\mathcal{R}| + |\Rm| |\mathcal{S}| + |\nabla^2 \Rm||\nabla \mathcal{P}| + |\nabla \Rm| |\nabla^2 \mathcal{P}| ),\\
|(D_t - \Delta) \mathcal{T}| & \le C( |\nabla^2 \Rm | | \mathcal{R}| + |\nabla \Rm| |\mathcal{S}| + |\Rm| |\mathcal{T}|\\
& \quad  + (|\nabla \Rm| | \Rm| + |\nabla^3 \Rm|) |\nabla \mathcal{P}| + |\nabla^2 \mathcal{R}| |\nabla^2 \mathcal{P}|),\\
\end{split}
\end{equation}
where $C = C(n) >0$.

\end{prop}

\begin{proof}

Using the evolution equation for $\Rm$, we have
\begin{equation*}
\begin{split}
(D_t - \Delta) \mathcal{R} & = \mathcal{Q} (\Rm, \Rm) \circ \mathcal{P} + \Rm \circ \Delta \mathcal{P} + 2 \nabla \Rm * \nabla \mathcal{P}.\\
\end{split}
\end{equation*}
The first inequality then follows immediately from Lemma \ref{q}.

We now compute the evolution equation for $\mathcal{S}$. First, note that
\begin{equation}\label{sid1}
\begin{split}
(D_t - \Delta) \mathcal{S} & = ([D_t -\Delta, \nabla] \Rm ) \circ \mathcal{P} + \nabla ((D_t - \Delta) \Rm) \circ \mathcal{P} \\
& \quad + \nabla^2 \Rm * \nabla \mathcal{P} + \nabla \Rm * \nabla^2 \mathcal{P}.\\
\end{split}
\end{equation}
For the first term, using the commutator (\ref{heatcomm}), we have
\begin{equation*}
[(D_t-\Delta), \nabla_a] R_{ijkl} = 2 R_{abdc} \Lambda^c_d \nabla_b R_{ijkl} + 2 R_{ab} \nabla_b R_{ijkl}. \\
\end{equation*}
As in the computation in Proposition 4.13 from \cite{kholonomy}, we have
\begin{equation*}
R_{abdc} \Lambda^c_d \nabla_b R_{mnkl} \mathcal{P}_{ijmn} = \Rm * \mathcal{S} + \nabla \Rm * \mathcal{R} * \mathcal{P},
\end{equation*}
which gives us
\begin{equation}
\begin{split}\label{sid2}
([D_t - \Delta, \nabla] \Rm) \circ \mathcal{P} & = \Rm * \mathcal{S} + \nabla \Rm * \mathcal{R} * \mathcal{P}.\\
\end{split}
\end{equation}
We then compute
\begin{equation*}
\begin{split}
\nabla ((D_t - \Delta) \Rm) & = \nabla \mathcal{Q} (\Rm, \Rm) \\
& = \nabla \Rm \circ \Rm + \Rm \circ \nabla \Rm + \nabla \Rm \# \Rm + \Rm \# \nabla \Rm \\
& = 2 \mathcal{Q}(\nabla \Rm, \Rm),\\
\end{split}
\end{equation*}
where we regard $\nabla \Rm$ as a one form with values in $\mathrm{Sym} (\Lambda^2 T^*M)$. Then, applying Lemma \ref{q} and combining the result in (\ref{sid2}) with (\ref{sid1}), we obtain the second inequality.

For the third inequality, we begin with the identity
\begin{equation*}
(D_t - \Delta) \mathcal{T} = \left( (D_t - \Delta) \nabla^2 \Rm \right) \circ \mathcal{P} + \nabla^2 \Rm * \nabla^2 \mathcal{P} + \nabla^3 \Rm * \nabla \mathcal{P}.
\end{equation*}
The first term can be rewritten as
 \begin{equation*}
 \begin{split}
 ((D_t - \Delta) \nabla^2 \Rm ) \circ \mathcal{P} & = ( [D_t - \Delta, \nabla] \nabla \Rm ) \circ \mathcal{P} + ( \nabla [D_t - \Delta, \nabla] \Rm) \circ \mathcal{P}\\
 & + ( \nabla \nabla (D_t - \Delta) \Rm) \circ \mathcal{P}.\\
 \end{split}
 \end{equation*}
Applying equation (\ref{heatcomm}) once again gives us
\begin{equation*}
\left( (D_t - \Delta) \nabla_a \nabla \Rm \right) \circ \mathcal{P} - \left( \nabla_a (D_t - \Delta) \nabla \Rm \right) \circ \mathcal{P} = \left( 2 R_{abdc} \Lambda^c_d \nabla_b \nabla \Rm + 2 R_{ab} \nabla_b \nabla \Rm \right) \circ \mathcal{P},
\end{equation*}
and we have
\begin{equation*}
 R_{abdc} \Lambda^c_d \nabla_b \nabla \Rm \circ \mathcal{P} = \Rm * \mathcal{T} + \nabla^2 \Rm * \mathcal{R} * \mathcal{P}
 \end{equation*}
 (again see \cite{kholonomy}, Proposition 4.13, also \cite{liuszek}). 
 We can see that
 \begin{equation*}
 \begin{split}
& ( \nabla[D_t - \Delta, \nabla] \Rm) \circ \mathcal{P} = \nabla ( ([D_t-\Delta, \nabla]\Rm) \circ \mathcal{P}) + ([D_t- \Delta, \nabla]\Rm) * \nabla \mathcal{P}\\
 & \quad = \nabla \Rm * \mathcal{S} + \Rm * \mathcal{T} + \Rm* \nabla \Rm * \nabla \mathcal{P} + \nabla^2 \Rm * \mathcal{R} * \mathcal{P} + \nabla \Rm * \mathcal{S} * \mathcal{P} \\
& \quad \quad + \nabla \Rm * \Rm * \nabla \mathcal{P} * \mathcal{P} + \nabla \Rm * \mathcal{R} * \nabla \mathcal{P} + \Rm * \nabla \Rm * \nabla \mathcal{P}\\
\end{split}
\end{equation*}
where we again use the facts that $\nabla \mathcal{R} = \mathcal{S} + \Rm * \nabla \mathcal{P}$ and $\nabla \mathcal{S} = \mathcal{T} + \nabla \Rm * \nabla \mathcal{P}$. Additionally,
\begin{equation*}
(\nabla \nabla (D_t - \Delta) \Rm ) \circ \mathcal{P} = 2 \mathcal{Q} (\nabla^2 \Rm, \Rm) \circ \mathcal{P} + 2 \mathcal{Q}(\nabla \Rm, \nabla \Rm ) \circ \mathcal{P}.
\end{equation*}
Combining the above identities and again applying Lemma \ref{q} to the last term, we obtain the third inequality.

\end{proof}

\subsection{Constructing a PDE-ODE system}
With an eye toward Theorem 1, we now organize the tensors $\nabla \mathcal{P}$, $\nabla^2 \mathcal{P}$, $\mathcal{R}$, $\mathcal{S}$, and $\mathcal{T}$ into groupings which satisfy a closed system of differential inequalities. Let 
\begin{equation*}
\begin{split}
\mathcal{X}  = \mathcal{T}^4 (T^*M) \oplus \mathcal{T}^5 (T^*M) \oplus \mathcal{T}^6(T^*M), \quad \mathcal{Y}  = \mathcal{T}^5 (T^*M) \oplus \mathcal{T}^6( T^*M),
\end{split}
\end{equation*}
and define families of sections $\mathbf{X} = \mathbf{X}(t)$ of $\mathcal{X}$ and $\mathbf{Y}=\mathbf{Y}(t)$ of $\mathcal{Y}$ for $t \in (0,T]$ by
\begin{equation}\label{systemdef}
\mathbf{X} = \left( \frac{\mathcal{R}}{t}, \frac{\mathcal{S}}{t^{1/2}}, \mathcal{T} \right), \qquad \mathbf{Y} = \left( \frac{\nabla \mathcal{P}}{\sqrt{t}}, \nabla^2 \mathcal{P} \right).
\end{equation}

\begin{prop}\label{systembound}
If $g(t)$ is a smooth solution to Ricci flow on $M \times [0,T]$ with $|\Rm| (x,t) < a/t$ for some $a>0$, then there exists a constant $C=C(a,n)>0$ depending such that $\mathbf{X}$ and $\mathbf{Y}$ satisfy
\begin{equation*}
|(D_t - \Delta) \mathbf{X}| \le C \left( \frac{1}{t} |\mathbf{X}| + \frac{1}{t^2} |\mathbf{Y}| \right), \quad |D_t \mathbf{Y}| \le C \left( |\mathbf{X}| + \frac{1}{t} |\mathbf{Y}| \right),
\end{equation*}
on $M \times (0,T]$.
\end{prop}

\begin{remark}
Inspection of the proof reveals that the constant $C$ in fact has the form $C=a \tilde{C}$, where $\tilde{C}$ depends only on $n$ and $\max \{ a, 1 \}$.
\end{remark}
This follows directly from Propositions \ref{schematicequationsp} and \ref{schematicequationsr} with the help of the following curvature bounds, which can be obtained from the classical estimates of Shi \cite{shideformmetric} with a simple rescaling argument.

\begin{prop}\label{curvatureestimates}
Suppose $(M,g(t))$ is a complete solution to Ricci flow for $t \in [0,T]$ which satisfies
\begin{equation*}
| \Rm | (x,t) \le \frac{a}{t}
\end{equation*}
for some constant $a > 0$. Then for each $m > 0$, there exists a constant $C = C(m,n)$ such that
\begin{equation*}
|\nabla^{(m)} \Rm | (x,t) \le \frac{a C}{t^{m/2+1}} (1 + a^{m/2}).
\end{equation*}

\end{prop}

\begin{proof}[Proof of Proposition \ref{systembound}]

Throughout this proof, $C$ will denote a constant which may change from line to line but depends only on $n$ and $a$. Using (\ref{pnorms}) in combination with the curvature estimates, we obtain

\begin{equation*}
\begin{split}
|D_t \mathbf{Y}| & \le  \frac{1}{2} t^{-3/2} |\nabla \mathcal{P}| + t^{-1/2} | D_t \nabla \mathcal{P}| + |D_t \nabla^2 \mathcal{P}|\\
& \le C t^{-1/2} |\mathcal{S}| + C |\mathcal{T}| + C t^{-3/2} |\nabla \mathcal{P}| + C t^{-1} |\nabla^2 \mathcal{P}|\\
& \le C |\mathbf{X}| + \frac{C}{t} |\mathbf{Y}|.\\
\end{split}
\end{equation*}

Applying the curvature estimates to the inequalities (\ref{rnorms}) for $\mathcal{R}$, $\mathcal{S}$, and $\mathcal{T}$, we get
\begin{equation*}
|(D_t - \Delta) \mathcal{R}| \le C t^{-1} |\mathcal{R}| + Ct^{-3/2} |\nabla \mathcal{P}| + Ct^{-1} |\nabla^2 \mathcal{P}|,
\end{equation*}
\begin{equation*}
|(D_t - \Delta) \mathcal{S}| \le C t^{-3/2} |\mathcal{R}| + C t^{-1} |\mathcal{S}| + Ct^{-2} |\nabla \mathcal{P}| + C t^{-3/2} |\nabla^2 \mathcal{P}|,
\end{equation*}
and
 \begin{equation*}
 |(D_t - \Delta) \mathcal{T}| \le C t^{-2} |\mathcal{R}| + C t^{-3/2} |\mathcal{S}| + Ct^{-1} |\mathcal{T}| + Ct^{-5/2} |\nabla \mathcal{P}| + t^{-2} |\nabla^2 \mathcal{P}|.
 \end{equation*}
Combining these equations, we have
 \begin{equation*}
 \begin{split}
 |(D_t - \Delta) \mathbf{X}| & \le t^{-1} |(D_t-\Delta) \mathcal{R}| + t^{-2} |\mathcal{R}| + t^{-1/2} |(D_t - \Delta) \mathcal{S}| + \frac{1}{2} t^{-3/2} |\mathcal{S}| \\
 & \quad + |(D_t - \Delta) \mathcal{T}|\\
 & \le Ct^{-2} |\mathcal{R}| + C t^{-3/2} |\mathcal{S}| + C t^{-1} |\mathcal{T}| + Ct^{-5/2} |\nabla \mathcal{P}| + Ct^{-2} |\nabla^2 \mathcal{P}|\\
 & \le Ct^{-1} |\mathbf{X}| + Ct^{-2} |\mathbf{Y}|,\\
 \end{split}
 \end{equation*} 
 as desired.
 \end{proof}

\section{A general uniqueness theorem for PDE-ODE systems}

We now aim to show that $\mathbf{X}$ and $\mathbf{Y}$ vanish using a maximum principle from \cite{huangtam} by adapting it to apply to a general PDE-ODE system. The following theorem is essentially a reformulation of Lemma 2.3 in \cite{huangtam} and Lemma 2.1 in \cite{liuszek}.

\begin{thm}\label{maxprinsystem}
Let $M = M^n$ and $\mathcal{X}$ and $\mathcal{Y}$ be finite direct sums of $T^k_l(M)$. There exists an $\epsilon = \epsilon (n) > 0$ with the following property: Whenever $g(t)$ is a smooth, complete solution to the Ricci flow on $M$ satisfying
\begin{equation*}
|\Rm| \le \frac{\epsilon}{t}
\end{equation*}
on $M \times (0,T]$, and $\mathbf{X} = \mathbf{X}(t)$ and $\mathbf{Y} = \mathbf{Y}(t)$ are families of smooth sections of $\mathcal{X}$ and $\mathcal{Y}$ satisfying
\begin{equation*}
|(D_t - \Delta) \mathbf{X}| \le \frac{C}{t} |\mathbf{X}| + \frac{C}{t^2} |\mathbf{Y}|, \quad | D_t \mathbf{Y} | \le  C |\mathbf{X}| + \frac{C}{t} | \mathbf{Y}|,
\end{equation*}
\begin{equation*}
D_t^k \mathbf{Y} = 0, \quad D_t^k \mathbf{X} = 0 \text{ for } k \ge 0 \text{ at } t=0,
\end{equation*}
and
\begin{equation*}
| \mathbf{X} | \le C t^{-l},
\end{equation*}
for some $C>0$, $l > 0$, then $\mathbf{X} \equiv 0$ and $\mathbf{Y} \equiv 0$ on $M \times [0,T]$.
\end{thm}

The key ingredient in the proof of Theorem \ref{maxprinsystem} is an the following scalar maximum principle due to Huang-Tam \cite{huangtam} (and its variant in \cite{liuszek}). Though the statement has been slightly changed from its appearance in \cite{huangtam}, the proof is nearly identical. We detail here the modifications we make for completeness.
\begin{prop}[c.f. \cite{huangtam}, Lemma 2.3 and \cite{liuszek}, Lemma 2.1]\label{maxprinscalar}

Let $M$ be a smooth $n$-dimensional manifold. There exists an $\epsilon = \epsilon (n) > 0$ such that the following holds: Whenever\ $g(t)$ is a smooth complete solution to the Ricci flow on $M \times [0,T]$ such that the curvature satisfies $|\Rm | \le \epsilon/t$ and $f \ge 0$ is a smooth function on $M \times [0,T]$ satisfying

\begin{enumerate}
\item $\left( \partial_t - \Delta \right) f (x,t) \le a t^{-1}\max_{0 \le s \le t} f(x,s)$,
\item $\partial_t^k \big|_{t=0} f = 0$ for all $k \ge 0$,
\item $\sup_{x \in M} f(x,t) \le C t^{-l}$ for some positive integer $l$ for some constant $C$,
\end{enumerate}
then $f \equiv 0$ on $M \times [0,T]$.

\end{prop}

\begin{proof}
For the time-being, we will assume $\epsilon >0$ is fixed and that $g(t)$ is a smooth, complete solution to Ricci flow on $M \times [0,T]$ satisfying $|\Rm| \le \epsilon/t$. We will then specify $\epsilon$ over the course of the proof.

As in \cite{huangtam} we may assume $T \le 1$. We will first show that for any $k > a$, there exists a constant $B_k$ such that
\begin{equation*}
\sup_{x \in M} f(x,t) \le B_k t^k.
\end{equation*}

Let $\phi$ be a cutoff function as in \cite{huangtam}, i.e., choose $\phi \in C^\infty ([0,\infty))$ such that $0 \le \phi \le 1$ and
\begin{equation*}
\phi(s) = 
\begin{cases}
1 & 0 \le s \le 1,\\
0 & 2 \le s,\\
\end{cases}
\quad -C_0 \le \phi' \le 0, \quad |\phi''| \le C_0,
\end{equation*}
for some constant $C_0> 0$. Then let $\Phi = \phi^m$ for $m>2$ to be chosen later and define $q =1 - \frac{2}{m}$. Then 
\begin{equation*}
\begin{split}
0 \ge \Phi'  \ge - C(m) \Phi^q, \quad |\Phi''| \le C(m) \Phi^q.
\end{split}
\end{equation*}
where $C(m) > 0$ is a constant depending only on $m$ (and on $C_0$). 

Fix a point $y \in M$. As in Lemma 2.2 of \cite{huangtam}, there exists some $\rho \in C^\infty (M)$ such that
\begin{equation*}
d_{g(T)} (x,y) + 1 \le \rho (x) \le C' (d_{g(T)} (x,y) + 1 ), \quad |\nabla_{g(T)} \rho |_{g(T)} + |\nabla_{g(T)}^2 \rho |_{g(T)} \le C',
\end{equation*}
where $C'$ is a positive constant depending only on $n$ and $\frac{\epsilon}{T}$. This function then also satisfies
\begin{equation*}
|\nabla \rho | \le C_1 t^{-c\epsilon}, \quad |\Delta \rho | \le C_2 t^{-1/2 - c\epsilon},
\end{equation*}
where $C_1, C_2$ are constants depending only on $n, T$ and $\epsilon$, and $c>0$ depends only on the dimension $n$. We may assume $\epsilon$ is small enough so that $c \epsilon < 1/4$. Let $\Psi (x) = \Psi_r(x) = \Phi(\rho (x)/r)$ for $r\gg 1$. Define also $\theta = \exp (-\alpha t^{1-\beta})$, where $\alpha >0$ and $0 < \beta < 1$. By the estimates on the derivatives of $\rho$, we have
\begin{equation*}
|\nabla \Psi| = r^{-1} |\Phi' (\rho/r)| | \nabla \rho | \le r^{-1} C(m) C_1 \Phi^q (\rho/r) t^{-c\epsilon} \le C(m) \Psi^q t^{-1/4}
\end{equation*}
and
\begin{equation*}
\begin{split}
| \Delta \Psi | & = |r^{-2} \Phi'' (\rho/r) | \nabla \rho |^2 + r^{-1} \Phi'(\rho/r) \Delta \rho |\\
&  \le r^{-2} C(m) \Phi^q (\rho/r) t^{-2 c \epsilon} + r^{-1} C(m) \Phi^q(\rho) t^{-1/2 - c\epsilon}\\
& \le C(m) \Psi^q t^{-3/4}.\\
\end{split}
\end{equation*}

For $k > a$, let $F = t^{-k} f$. Then $F$ satisfies
\begin{equation*}
\begin{split}
(\partial_t - \Delta) F & = -k t^{-k-1} f + t^{-k} (\partial_t - \Delta) f \\
& \le -kt^{-k-1} f (x,t) + a t^{-k-1} \max_{0 \le s \le t} f(x,s)\\
\end{split}
\end{equation*}
and $F \le C t^{-l-k}$.

Let $H = \theta \Psi F$ and suppose that $H$ attains a positive maximum at the point $(x_0,t_0)$. Then, at this point, we have $\Psi > 0$ and both $(\partial_t - \Delta )H \ge 0$ and $\nabla H =0$. Since $\nabla H = 0$, we have
\begin{equation*}
 \nabla \Psi \cdot \nabla F = -\frac{F |\nabla \Psi |^2}{\Psi}.
\end{equation*}

Additionally, since $\Psi$ is independent of time, 
\begin{equation*}
 \theta(s) F(x_0,s) \le \theta (t_0) F(x_0,t_0)
\end{equation*}
for all $s \le t_0$. Because $\theta$ is decreasing, we have
\begin{equation*}
 s^{-k} f(x_0,s) = F(x_0,s) \le F(x_0,t_0) = t_0^{-k} f(x_0,t_0)
\end{equation*}
for $s \le t_0$. which in turn implies
\begin{equation*}
\max_{0 \le s \le t_0} f(x,s) = f(x,t_0).
\end{equation*}
Thus, at $(x_0,t_0)$ we have
\begin{equation*}
 (\partial_t - \Delta)F \le (-k + a)t_0^{-1} F \le 0.
\end{equation*}

Thus at $(x_0,t_0)$ we have
\begin{equation*}
\begin{split}
\Delta H & = \theta F \Delta \Psi + \theta \Psi \Delta F + 2 \theta \nabla F \cdot \nabla \Psi \\
& = \theta F \Delta \Psi + \theta \Psi \Delta F - 2\theta \frac{F |\nabla \Psi |^2}{\Psi} \\
& \ge- C(m) \theta F \Psi^q t_0^{-3/4} - C(m) \theta F \Psi^{2q-1} t_0^{-3/4} + \theta \Psi \Delta F
\end{split}
\end{equation*}
and
\begin{equation*}
\partial_t H =  -\alpha(1 - \beta)t_0^{-\beta} \theta \Psi F + \theta \Psi \partial_t F.
\end{equation*}
We can then compute
\begin{equation*}
\begin{split}
0 & \le (\partial_t - \Delta) H \\
& \le \theta \Psi (\partial_t - \Delta)F- \alpha (1-\beta) t_0^{-\beta} \theta \Psi F + C(m)\theta \Psi^q F t_0^{-3/4} + C(m) \theta \Psi^{2q-1}F t_0^{-3/4} \\
& \le -\alpha (1 - \beta) t_0^{-\beta}\theta \Psi F + C(m) \theta (\Psi F)^q  t_0^{-3/4 - (1-q)(l+k)}\\
& \quad + C(m) \theta (\Psi F)^{2q-1}t_0^{-3/4 - (2-2q)(l+k)} .\\
\end{split}
\end{equation*}

We now choose $m$ and $\beta$ so that the powers of $t_0$ in the denominators of the last two terms are less than $\beta$. We take $\beta$ to be $7/8$ (any $\beta \in (3/4,1)$ will do). Recalling that $q = 1-2/m$, we choose $m$ large enough so that $7/8 > 3/4 + (1-q)(l+k)$ and $7/8 > 3/4 + (2-2q)(l+k)$. Then
\begin{equation*}
\frac{\alpha}{8} \Psi F = \alpha (1- \beta) \Psi F \le C(m) \left( (\Psi F)^q + (\Psi F)^{2q-1} \right).
\end{equation*}
Finally, we choose $\alpha$ large enough so that $\alpha > 16 C(m)$. Then
\begin{equation*}
2\Psi F \le (\Psi F)^q + (\Psi F)^{2q-2},
\end{equation*}
implying that $(\Psi F) (x_0,t_0) \le 1$, and hence $H \le 1$ everywhere. In particular, for any $x \in \{ \rho \le r \}$, $f(x,t) = t^k F(x,t) \le e^\alpha t^k :=B_kt^k$. Sending $r$ to infinity then proves that $f(x,t) \le B_k t^k$. 

Next, again as in \cite{huangtam}, we define the function $\eta (x,t) = \rho(x) \exp \left( \frac{2 C_2}{1-b} t^{1-b} \right)$ for $b>1$. Since $|\Delta \rho| \le C_2 t^{-b}$, we have
\begin{equation*}
\left( \partial_t - \Delta \right) \eta > 0, \quad \partial_t \eta > 0.
\end{equation*}
Let $F = t^{-a} f$. Fix $\delta > 0$ and consider the function $F - \delta \eta - \delta t$. Note that by our previous argument, $F \le C t^2$, and in particular is bounded. For some $t_1 > 0$ depending on $\delta$ and $c$, $F - \delta t < 0$ for $t \le t_1$ and for $t \ge t_1$, $F-\delta \eta < 0$ outside some compact set. So, if $F- \delta \eta - \delta t$ is ever positive, there must exist some $(x_0,t_0) \in M \times (0,T]$ at which it attains a positive maximum. Because $-\delta \eta-\delta t$ is decreasing in time, for any $s < t_0$ from the inequality
\begin{equation*}
F(x_0,s)-\delta \eta(x_0,s) - \delta s \le F(x_0,t_0) - \delta \eta (x_0,t_0) - \delta t_0,
\end{equation*}
we conclude
\begin{equation*}
 F(x_0, s) \le F(x_0,t_0).
\end{equation*}
As in our previous argument, this implies that $f(x_0,t_0) = \max_{0\le s \le t_0} f(x_0,s)$, so that at $(x_0,t_0)$
\begin{equation*}
(\partial_t -\Delta) (F-\delta \eta - \delta t) < 0,
\end{equation*}
a contradiction. Thus, for any $\delta > 0$, $F - \delta \eta - \delta t \le 0$. Taking $\delta \to 0$ then implies that $F = 0$.
\end{proof}

We can now prove Theorem \ref{maxprinsystem}.
\begin{proof}[Proof of Theorem \ref{maxprinsystem}]
For $k > 0$ to be determined later, define the functions $F$ and $G$ on $M \times [0,T]$ by $F = t^{-k} |\mathbf{X}|^2$, $G = t^{-(k+1)} |\mathbf{Y}|^2$ for $t \in (0,T]$ and $F(x,0) = G(x,0) = 0$. From the assumption that $D_t^l\mathbf{X} = D_t^l\mathbf{Y} = 0$ for all $l\ge0$, it follows that both $F$ and $G$ are smooth on $M \times [0,T]$ and that $\partial_t^l F = \partial_t^l G =0$ for all $l \ge 0$.

We have
\begin{equation*}
\begin{split}
(\partial_t - \Delta) F & =-k t^{-(k+1)} |\mathbf{X}|^2 + 2 t^{-k} \langle (D_t - \Delta) \mathbf{X}, \mathbf{X} \rangle - 2 t^{-k} | \nabla \mathbf{X}|^2\\
& \le -k t^{-(k+1)} |\mathbf{X}|^2 + 2 t^{-k} | (D_t - \Delta) \mathbf{X}| |\mathbf{X}| \\
& \le t^{-(k+1)} (2C - k) |\mathbf{X}|^2 + 2 C t^{-(k+2)} |\mathbf{X}| |\mathbf{Y}| \\
& \le t^{-1} (3C - k) F + Ct^{-2} G\\
\end{split}
\end{equation*}
and
\begin{equation*}
\begin{split}
\partial_t G & = -(k+1) t^{-(k+2)} |\mathbf{Y}|^2 + 2 t^{-(k+1)} \langle D_t \mathbf{Y}, \mathbf{Y} \rangle \\
& \le (2C - k - 1) t^{-(k+2)} |\mathbf{Y}|^2 + 2 C t^{-(k+1)} |\mathbf{X}| |\mathbf{Y}|\\
 & \le C F + t^{-1} (3C- k - 1) G.\\
\end{split}
\end{equation*}
Choosing $k > 3C$, this becomes
\begin{equation*}
(\partial_t - \Delta) F \le t^{-2} C G, \qquad \partial_t G \le C F.
\end{equation*}
In particular this implies that
\begin{equation*}
G(x,t) \le C t \max_{0 \le s \le t} F(x,s),
\end{equation*}
and therefore
\begin{equation*}
(\partial_t - \Delta) F \le t^{-1} C^2 \max_{ 0 \le s \le t } F (x,s).
\end{equation*}
By our assumption on $\mathbf{X}$, $F \le C t^{-2l-k}$. Thus $F$ satisfies the hypotheses of Proposition \ref{maxprinscalar}, and must vanish identically. We then conclude that $G$, hence $\mathbf{Y}$, vanishes as well.
\end{proof}

\section{Proof of Theorem \ref{main}}

We are now almost ready to prove Theorem \ref{main}. We just need to first verify that $\mathbf{X}$ and $\mathbf{Y}$ satisfy the last major remaining hypothesis of Theorem \ref{maxprinsystem}, that is, that all time derivatives of $\mathbf{X}$ and $\mathbf{Y}$ vanish at $t=0$.

\subsection{Vanishing of time derivatives}
We begin by recording a standard commutator formula, which is in fact valid (with obvious modifications) for any family of smooth metrics.

\begin{prop}\label{inductioncomm}
Let $(M,g(t))$ be a smooth solution to the Ricci flow for $t \in [0,T]$.
 Then, for any $l \ge 1$, the formula
\begin{equation}\label{multicomm}
[D_t, \nabla^{(l)}] \mathcal{A} = \sum_{k=1}^l \nabla^{(k-1)} [D_t,\nabla] \nabla^{(l-k)} \mathcal{A}
\end{equation}
is valid for any smooth family of tensor fields $\mathcal{A}$ on $M \times [0,T]$.
\end{prop}

\begin{proof}
We proceed by induction on $l$. The base case, $l=1$, is trivial. Now, suppose that (\ref{multicomm}) holds for $l \le m$ for some $m \ge 1$. Then,
\begin{equation*}
\begin{split}
[D_t, \nabla^{(m+1)}] \mathcal{A} & = D_t \nabla^{(m+1)} \mathcal{A} - \nabla^{(m+1)} D_t \mathcal{A}\\
& = [D_t, \nabla^{(m)}] \nabla \mathcal{A} + \nabla^{m} D_t \nabla \mathcal{A} - \nabla^{(m+1)} D_t \mathcal{A} \\
& = [D_t, \nabla^{(m)}] \nabla \mathcal{A} + \nabla^{(m)}[D_t, \nabla] \mathcal{A} \\
& = \sum_{k=1}^m \nabla^{(k-1)} [D_t, \nabla] \nabla^{(m-k)} ( \nabla \mathcal{A} ) + \nabla^{(m)} [D_t, \nabla] \mathcal{A} \\
& = \sum_{k=1}^{m+1} \nabla^{(k-1)} [D_t, \nabla] \nabla^{(m+1-k)} \mathcal{A},\\
\end{split}
\end{equation*}
as desired.
\end{proof}

Now we argue inductively that $D_t^k \mathbf{X} = 0$ and $D_t^k \mathbf{Y}=0$ at $t=0$.

\begin{prop}\label{timederivativesvanish}
Let $M=\hat{M} \times \check{M}$ be a smooth manifold and $g(t)$ be a smooth, complete solution to the Ricci flow such that $g(0)$ splits as a product. Define $\mathcal{P}$ and $\mathcal{R}$ as in Section 2. The following equations hold at $t=0$ for all $k, l \ge 0$:

\begin{equation*}
D_t^k \nabla^{(l)} \mathcal{R} = 0,\quad D_t^k \nabla^{(l+1)} \mathcal{P} = 0.
\end{equation*}

\end{prop}

\begin{proof}
We proceed by induction on $k$, beginning with the base case $k=0$. Because the metric splits as a product initially, at $t=0$ we have $\nabla^{(l)} \hat{P}\equiv\nabla^{(l)} \check{P}\equiv 0$ for all $l \ge 0$ and
\begin{equation*}
R(\hat{P} (\cdot), \check{P} (\cdot), \cdot , \cdot ) \equiv 0.
\end{equation*}
From this we get that, for any $X, Y, Z, W \in TM$,
\begin{equation*}
\mathcal{R} (X^* \wedge Y^*) (Z,W) = 2 R( \hat{P} X, \check{P}Y, W, Z) + 2 R(\check{P}X, \hat{P} Y, W, Z) = 0.
\end{equation*}

Combining these facts, we conclude
\begin{equation*}
\nabla^{(l+1)} \mathcal{P} \equiv 0, \quad \nabla^{(l)} \mathcal{R} \equiv 0, \quad \nabla^{(l)} \mathcal{R}^* \equiv 0,
\end{equation*}
at $t=0$, where $\mathcal{R}^* = \mathcal{P} \circ \Rm$ denotes the adjoint of $\mathcal{R}$ with respect to $g$.

Now starting the induction step, suppose that for some $k \ge 0$, for all $l \ge 0$ and any $m \le k$,
\begin{equation*}
D^m_t \nabla^{(l+1)} \mathcal{P} = 0, \quad D^m_t \nabla^{(l)} \mathcal{R} = 0,
\end{equation*}
hence also $D_t^m \nabla^{(l)} \mathcal{R}^* = 0$. Recall that
\begin{equation*}
(D_t - \Delta) \Rm = \mathcal{Q} (\Rm, \Rm) .
\end{equation*}
As in \cite{kholonomy}, Lemma 4.9, $\mathcal{Q}(\Rm, \Rm) \circ \mathcal{P} = \mathcal{R} * \mathcal{U}_1 + \mathcal{R}^* * \mathcal{U}_2$, where $\mathcal{U}_1$ and $\mathcal{U}_2$ are smooth families of tensors on $M$. Thus we can compute
\begin{equation*}
\begin{split}
D_t \mathcal{R} & = (D_t \Rm ) \circ \mathcal{P} + \Rm \circ (D_t \mathcal{P}) \\
& = (\Delta \Rm) \circ \mathcal{P} + \mathcal{Q} (\Rm, \Rm) \circ \mathcal{P},\\
\end{split}
\end{equation*}
and thus
\begin{equation}\label{dtr}
\begin{split}
D_t^{k+1} \mathcal{R} & = D_t^k \left( (\Delta \Rm)\circ \mathcal{P} \right) + D_t^k \left( \mathcal{Q} (\Rm, \Rm) \circ \mathcal{P} \right).\\
\end{split}
\end{equation}

Because
\begin{equation*}
\Delta \mathcal{R} = (\Delta \Rm) \circ \mathcal{P} + \Rm \circ \Delta \mathcal{P} + 2 \nabla_i \Rm \circ \nabla_i \mathcal{P},
\end{equation*}
by the induction hypothesis $D_t^k ((\Delta \Rm) \circ \mathcal{P}) \equiv 0$ at $t=0$. Similarly,
\begin{equation*}
D_t^k \big( \mathcal{Q} (\Rm, \Rm) \circ \mathcal{P} \big) =D_t^k (\mathcal{R} * \mathcal{U}_1) + D_t^k (\mathcal{R}^* * \mathcal{U}_2) = 0.
\end{equation*}
We conclude that $D_t^{k+1} \mathcal{R} \equiv 0$, and thus $D_t^{k+1} \mathcal{R}^* \equiv 0$.

Now, using the commutator from equation (\ref{dtcomm}) and Proposition \ref{inductioncomm}, for any $l > 0$ we have
\begin{equation*}
\begin{split}
D_t \nabla^{(l)} \mathcal{R} & = \sum_{m=1}^l \nabla^{(m-1)} [D_t, \nabla] \nabla^{(l-m)}\mathcal{R} + \nabla^{(l)} D_t \mathcal{R}\\
& = \sum_{m=1}^l \nabla^{(m-1)} \left( \nabla \Rm * \nabla^{(l-m)} \mathcal{R} + \Rm * \nabla^{(l-m+1)} \mathcal{R} \right) + \nabla^{(l)} D_t \mathcal{R},\\
\end{split}
\end{equation*}
and thus
\begin{equation*}
D_t^{k+1} \nabla^{(l)} \mathcal{R} = \sum_{m=1}^l D_t^{k}\nabla^{(m-1)} \left( \nabla \Rm * \nabla^{(l-m)} \mathcal{R} + \Rm * \nabla^{(l-m+1)} \mathcal{R} \right) + D_t^k \nabla^{(l)} D_t \mathcal{R}.
\end{equation*}
Expanding using the product rule and applying the induction hypothesis, all terms in the first sum vanish at $t=0$. For the remaining term, we again use the evolution equation for $\mathcal{R}$. We have
\begin{equation*}
\begin{split}
D_t^k \nabla^{(l)} D_t \mathcal{R} & = D_t^k \nabla^{(l)} \left( (\Delta \Rm) \circ \mathcal{P} + \mathcal{Q}(\Rm, \Rm) \circ \mathcal{P} \right).\\
\end{split}
\end{equation*}
As before, rewriting $\mathcal{Q}(\Rm, \Rm) \circ \mathcal{P}$ in terms of $\mathcal{R}$ and $ \mathcal{R}^*$ and expanding using the product rule, it follows that $D_t^k \nabla^{(l)} D_t \mathcal{R} \equiv 0$ at $t=0$.

We now move on to the derivatives of $\mathcal{P}$. Recall that
\begin{equation*}
D_t \nabla \mathcal{P} = [D_t, \nabla] \mathcal{P} = \Rm * \nabla \mathcal{P} + \mathcal{P} * \mathcal{S}.
\end{equation*}
Applying this in combination with Proposition \ref{inductioncomm}, we get, for any $l \ge 1$,
\begin{equation*}
\begin{split}
D_t^{k+1} \nabla^{(l)} \mathcal{P} & = \sum_{m=1}^l D_t^k \nabla^{(m-1)}  [D_t, \nabla] \nabla^{(l-m)} \mathcal{P} + D_t^k \nabla^{(l)} D_t \mathcal{P}\\
& = \sum_{m=1}^{l-1} D_t^k \nabla^{(m-1)} \big( \nabla \Rm * \nabla^{(l-m)} \mathcal{P} + \Rm * \nabla^{(l-m+1)} \mathcal{P} \big) \\
& \quad + D_t^k \nabla^{(l-1)} [D_t, \nabla] \mathcal{P} + D_t^k \nabla^{(l)} D_t \mathcal{P}.\\
\end{split}
\end{equation*}
As before, every term in the first sum vanishes by the induction hypothesis, while the final term vanishes because $D_t \mathcal{P} \equiv 0$. Finally we can see that
\begin{equation*}
D_t^k \nabla^{(l-1)} [D_t, \nabla] \mathcal{P} = D_t^k \nabla^{(l-1)} (\Rm * \nabla \mathcal{P} + \mathcal{P} * \mathcal{S} ),
\end{equation*}
and because $\mathcal{S} = (\nabla \Rm) \circ \mathcal{P} = \nabla \mathcal{R} + \Rm * \nabla \mathcal{P}$, $D_t^k \nabla^{(l-1)} [D_t, \nabla] \mathcal{P} \equiv 0$ at $t=0$.
This completes the proof.
\end{proof}

\subsection{Preservation of product structures}
In the proof of Theorem \ref{main}, we will use the operator $\mathcal{F}: \Lambda^2 T^*M \to \Lambda^2 T^*M$ defined by
\begin{equation*}
\mathcal{F} \omega (X,Y) = \omega (\hat{P} X, \check{P} Y) - \omega ( \check{P} X, \hat{P} Y).
\end{equation*}
(See, for example, Section 2.2 of \cite{kahlerity}.) Observe that
\begin{equation*}
\begin{split}
\mathcal{P} \circ \mathcal{F} \omega (X,Y) & = \mathcal{F} \omega (\hat{P} X, \check{P} Y) + \mathcal{F} \omega (\check{P} X, \hat{P} Y) \\
& = \omega ( \hat{P}^2 X, \check{P}^2 Y) - \omega( \check{P} \hat{P} X, \hat{P} \check{P} Y) +  \omega ( \hat{P} \check{P} X, \check{P} \hat{P} Y) - \omega (\check{P}^2 X, \hat{P}^2 Y)\\
& = \omega ( \hat{P} X, \check{P} Y) - \omega (\check{P} X, \hat{P} Y).\\
\end{split}
\end{equation*}
Therefore $\mathcal{P} \circ \mathcal{F} \equiv \mathcal{F}$.

\begin{proof}[Proof of Theorem 1]
We have shown in Propositions \ref{systembound} and \ref{timederivativesvanish} that the system $\mathbf{X}, \mathbf{Y}$ satisfies the first two hypotheses of Theorem \ref{maxprinsystem}. Additionally, the curvature bounds from Proposition \ref{curvatureestimates} imply that $|\mathbf{X}| \le C t^{-2}$. Thus, $\mathbf{X} \equiv 0$ and $\mathbf{Y} \equiv 0$ on $M \times [0,T]$. In particular, we know that $\mathcal{R} \equiv 0$ and $\nabla \mathcal{P} \equiv 0$ on $M \times [0,T]$.

We claim that $\nabla \hat{P} \equiv \nabla \check{P} \equiv 0$ and $\partial_t \hat{P} \equiv \partial_t \check{P} \equiv 0$. Similar to the proof of Lemma 7 in \cite{kahlerity}, if we define $W = \nabla \hat{P}$, then
\begin{equation*}
\begin{split}
D_t W_{ai}^j = [D_t, \nabla_a] \hat{P}^j_i & = \nabla_p R_{pai}^c \hat{P}^j_c - \nabla_p R_{pab}^j \hat{P}^b_i + R_{a}^c W_{ci}^j.
\end{split}
\end{equation*}
Note that the first two terms combine to give 
\begin{equation*}
\begin{split}
&\langle \nabla_{e_p} R (e_p, e_a) \check{P} e_i, \hat{P} e_j \rangle - \langle \nabla_{e_p} R (e_p, e_a) \hat{P} e_i, \check{P} e_j \rangle \\
& \quad = \frac{1}{2} \left( \nabla_{e_p} \Rm (\check{P} e_i^* \wedge \hat{P} e_j^*) ( e_p, e_a) - \nabla_{e_p} \Rm( \hat{P} e_i^* \wedge \check{P} e_j^*) (e_p, e_a) \right)\\
& \quad = -\frac{1}{2} \nabla_{e_p} \Rm \circ \mathcal{F} (e_i^* \wedge e_j^*) (e_p,e_a). \\
\end{split} 
\end{equation*}
But, since $\mathcal{P} \circ \mathcal{F} = \mathcal{F}$,
\begin{equation*}
\nabla \Rm \circ \mathcal{F} = \nabla \Rm \circ \mathcal{P} \circ \mathcal{F} = \mathcal{S} \circ \mathcal{F} = 0,
\end{equation*}
so $D_t W_{ai}^j = R^c_a W^j_{ci}$. Thus, for any point $x \in M$, the function $f(t) =| \nabla \hat{P}|^2 (x,t)$ satisfies
\begin{equation*}
f'(t) \le Cf
\end{equation*}
for some $C$ depending on $x$. Since $f(0) = 0$, $f$ is identically zero. Thus $\hat{P}$ (and similarly $\check{P}$) remain parallel.

Hence, we have
\begin{equation*}
R (\cdot, \cdot, \hat{P} ( \cdot), \check{P} (\cdot)) = 0,
\end{equation*}
which implies that $\Rc \circ \hat{P} = \hat{P} \circ \Rc$ and $\Rc \circ \check{P} = \check{P} \circ \Rc$, and thus, from (\ref{pdefs}), $\partial_t \hat{P} = \partial_t \check{P} = 0$ on $[0,T]$. Theorem \ref{main} follows.

\end{proof}

\end{document}